\documentclass[a4paper,12pt]{article}

\usepackage{tikz}

\usetikzlibrary{intersections, calc, patterns}

\newcounter{diagram}

\newenvironment{diagram}
   {\stepcounter{diagram}\par\smallskip\noindent\begin{minipage}{\linewidth}\centering}
   {\par Figure~\thediagram\end{minipage}\par\smallskip}

\usepackage{amsmath,amssymb}
\usepackage{amsfonts}
\usepackage{amsthm}
\usepackage{enumitem}
\usepackage{float}
\usepackage{indentfirst}
\usetikzlibrary{arrows}
\usepackage{comment}

\theoremstyle{plain}
\newtheorem{thm}{Theorem}[section]
\newtheorem{dfn}[thm]{Definition}

\newtheorem{cor}[thm]{Corollary}
\newtheorem{lem}[thm]{Lemma}

\theoremstyle{definition}
\newtheorem{ex}[thm]{Example}
\newtheorem{rem}[thm]{Remark}
\newtheorem{ass}[thm]{Assumption}

\usepackage{bm}

\numberwithin{equation}{section}

\title{$A$-hypergeometric Series with Parameters in the Core}
\author{Mao Nagamine
\thanks{This work was supported by JST SPRING, Grant Number JPMJSP2119.}}
\date{}

\begin{document}
\maketitle
\begin{abstract}
In this paper, we discuss the computational approach to the results established by Okuyama and Saito \cite{log19}. Although their results are often difficult to compute, we prove that, when the negative support of a fake exponent $\bm{v}$ with respect to a generic weight $\bm{w}$ is included in a certain set, solutions can be computed using only the reduced Gr\"{o}bner basis, and we can construct all $A$-hypergeometric series with exponent $\bm{v}$ in the direction $\bm{w}$ by Frobenius's method. As an example, we describe the Aomoto-Gel'fand system of type $3 \times 3$ in details.\par
\end{abstract}

\section{Introduction}
$A$-hypergeometric systems are systems of differential equations defined by Gel'fand, Kapranov, and Zelevinsky (GKZ). Frobenius’s method is a technique for constructing logarithmic series solutions by perturbing an exponent of a generic series solution \cite{frobenius}. This method is effective for constructing logarithmic solutions when the parameters are not generic. Such situations arise, for example, in the context of mirror symmetry (for example, see \cite{hosono}, \cite{Jan}). Okuyama and Saito demonstrated in \cite{log19,saito5} that this method can be applied to A-hypergeometric systems. We follow the notation in \cite{log19}. Let $A=[\bm{a}_{1},\ldots,\bm{a}_{n}]=[a_{ij}]$ be a $d \times n$-matrix of rank $d$ with coefficients in $\mathbb{Z}$. Throughout this paper, we assume that all $\bm{a}_{j}$ belong to one hyperplane off the origin in $\mathbb{Q}^d$. Let $\mathbb{N}$ be the set of nonnegative integers, i.e., $\mathbb{N}=\{0,1,2,\ldots\}$. Let $I_{A}$ denote the toric ideal in the polynomial ring $\mathbb{C}[\bm{\partial}]=\mathbb{C}[\partial_{1},\ldots,\partial_{n}]$, i.e.,
\begin{equation*}
    I_{A} = \langle
    \partial^{\bm{u}}-\partial^{\bm{v}} \mid
    A\bm{u}=A\bm{v},\;\bm{u},\bm{v}\in \mathbb{N}^{n}\rangle
    \subset \mathbb{C}[\bm{\partial}].
    \end{equation*}
Here and hereafter we use the multi-index notation; for example, $\partial^{\bm{u}}$ means $\partial^{u_{1}}_{1}\cdots \partial^{u_{n}}_{n}$ for $\bm{u} = (u_{1},\ldots,u_{n})^{T}$. Given a column vector $\bm{\beta}=(\beta_{1},\ldots,\beta_{d})^{T} \in \mathbb{C}^{d}$, let $H_{A}(\bm{\beta})$ denote the left ideal of the Weyl algebra
\begin{equation*}
    D = \mathbb{C}\langle \bm{x},\bm{\partial}\rangle
    = \mathbb{C} \langle x_{1},\ldots,x_{n},\partial_{1},\ldots,\partial_{n}\rangle
\end{equation*}
generated by $I_{A}$ and
\begin{equation*}
    \sum_{j=1}^{n}a_{ij}\theta_{j}-\beta_{i} \quad\quad (i=1,\ldots,d),
\end{equation*}
where $\theta_{j}=x_{j}\partial_{j}$. The quotient $M_{A}(\bm{\beta}) = D / H_{A}(\bm{\beta})$ is called the \textit{A-hypergeometric system with parameter} $\bm{\beta}$, and a formal series annihilated by $H_{A}(\bm{\beta})$ an \textit{A-hypergeometric series with parameter} $\bm{\beta}$.\par
Logarithm-free solutions to $M_{A}(\bm{\beta})$ were constructed by Gel'fand et al.\cite{gelfand} for a generic parameter $\bm{\beta}$. Okuyama and Saito applied Frobenius's method to the $A$-hypergeometric systems \cite{log19}. By this method, we might be able to construct bases of solution space for a nongeneric parameter. Logarithmic coefficients of $A$-hypergeometric series solutions are known to be polynomials in $\log{x^{\bm{b}}}\;(\bm{b}\in L)$ \cite[Proposition 5.2]{saito5}, where
\begin{equation*}
    L:= \mathop{\mathrm{Ker}}\nolimits_{\mathbb{Z}}A=\{ \bm{u}\in \mathbb{Z}^{n} \mid A\bm{u}=\bm{0}\}.
\end{equation*}\par
In Section 2, we recall the main theorem of \cite{log19} for constructing logarithmic series solutions of $A$-hypergeometric systems.\par 
In Section 3, fixing a generic weight $\bm{w}$, we consider the initial ideal $\mathrm{in}_{\bm{w}}I_{A}$ and the triangulation determined by $\bm{w}$. By \cite[Corollary 3.2.3]{log19}, $\bm{v}$ is a fake exponent of $M_{A}(\bm{\beta})$ with respect to $\bm{w}$ if and only if $A\bm{v}=\bm{\beta}$ and there exists a standard pair $(\bm{a},\sigma)$ of $\mathrm{in}_{\bm{w}}I_{A}$ such that $v_{j}=a_{j}$ for all $j\notin \sigma$. Then, we obtain the main theorem (Theorem \ref{mainth}) in this paper. This theorem says that, if the negative support of a fake exponent $\bm{v}$ is included in a certain set, then no coefficients of the perturbed hypergeometric series in Frobenius's method
\begin{equation*}
    F(\bm{x},\bm{s})=\sum_{u\in L'} a_{u}(s)x^{\bm{v}+B\bm{s}+\bm{u}}
\end{equation*}
have poles at $\bm{s}=\bm{0}$, where $B$ is a basis of $L(=\mathrm{Ker}_{\mathbb{Z}}A)$, and $L'$ is a certain subset of $L$, and then $\bm{v}$ is an exponent of $H_{A}(\bm{\beta})$. In this particular case, a basis of the solution space can be obtained solely by computation of the reduced Gröbner basis.
Finally, we applly this theorem to the case when the triangulation $\Delta_{\bm{w}}$ is unimodular, and $\bm{\beta}$ belongs to the core
\begin{equation*}
    \mathrm{Core}(\Delta_{\bm{w}}):=\mathbb{C}\bigcap_{F:\mathrm{facet\; of\; }\Delta_{\bm{w}}} F.
\end{equation*}
In this case, we can construct the series solution space in the direction $\bm{w}$ to $H_{A}(\bm{\beta})$.\par
We construct the bases of solution spaces in a different way from Section 3.6 in \cite[Section 3.5]{SST}. In \cite{SST}, Saito et al. constructed the solution spaces by considering the limit from generic parameters to a non-generic parameter. While the method in \cite[Section 2.6]{SST} solves the problem algorithmically, Frobenius’s method provides a non-algorithmic approach.\par
In Section 4, we describe the Aomoto-Gel'fand system of type $3\times 3$ in details. In addition, we discuss application of the main theorem to K3 surfaces.

\section{Logarithmic Hypergeometric Series}
In this section, we recall some notations and statements in \cite{log19} for construction of logarithmic series solutions of $A$-hypergeometric systems.\par
For $\bm{v}=(v_{1},\ldots,v_{n})^{T}\in \mathbb{C}^{n}$, its support $\mathrm{supp}(\bm{v})$ and its \textit{negative support} $\mathrm{nsupp}(\bm{v})$ are defined as
\begin{align*}
\mathrm{supp}(\bm{v})&:=\{j\in \{1,\ldots , n \}\mid v_{j}\neq 0 \},\\
\mathrm{nsupp}(\bm{v})&:= \{j\in \{1,\ldots , n \}\mid v_{j}\in \mathbb{Z}_{<0} \},
\end{align*}
respectively.\par
For $\bm{v}\in\mathbb{C}^{n}$ and $\bm{u}\in\mathbb{N}^{n}$, set
$$
[\bm{v}]_{\bm{u}}:=\prod_{j=1}^{n}v_{j}(v_{j}-1)\cdots (v_{j}-u_{j}+1).
$$
Note that we can uniquely write $\bm{u}\in\mathbb{Z}^{n}$ as the sum $\bm{u}=\bm{u}_{+}-\bm{u}_{-}$ with $\bm{u}_{+},\bm{u}_{-}\in\mathbb{N}^{n}$ and $\mathrm{supp}(\bm{u}_{+})\cap \mathrm{supp}(\bm{u}_{-})=\emptyset$.\par
Let $B=\{\bm{b}^{(1)},\ldots,\bm{b}^{(h)}\}\subset L$. We write the same symbol $B$ for the $n\times h$ matrix $(\bm{b}^{(1)},\ldots,\bm{b}^{(h)})$.\par
Set
\begin{equation*}
    \mathrm{supp}(B) := \bigcup_{k=1}^{h} \mathrm{supp}(\bm{b}^{(k)})\subset \{1,\ldots,n\},
\end{equation*}
which means the set of all labels for nonzero rows in $B$.\par
Let $\bm{s}=(s_{1},\ldots,s_{h})^{T}$ be indeterminates, and let
\begin{equation*}
    (B\bm{s})_{j}=\sum_{k=1}^{h} b_{j}^{(k)}s_{k} \in \mathbb{C}[\bm{s}]:= \mathbb{C}[s_{1},\ldots,s_{h}]
\end{equation*}
for $j=1,\ldots,n$. Set
\begin{equation*}
    (B\bm{s})_{J}:=\prod_{j\in J}(B\bm{s})_{j}\in \mathbb{C}[\bm{s}]
\end{equation*}
for $J\subset \{1,\ldots,n\}$. Note that $(B\bm{s})_{j}=0$ if $j \notin \mathrm{supp}(B)$, hence we have $(B\bm{s})_{J}=0$ if $J \not\subset \mathrm{supp}(B)$.
%\begin{lem}\textup{{\cite[Lemma 2.1]{log19}}}\label{2.1}
    %Let $B = \{\bm{b}^{(1)}\ldots,\bm{b}^{(h)}\} \subset L,\bm{u},\bm{u'}\in L$ and $\bm{v}\in \mathbb{C}$. Let $\bm{s}=(s_{1},\ldots,s_{h})^{T}$ be indeterminates. Then $[\bm{v}+B\bm{s}+\bm{u}]_{\bm{u'}_{+}}\neq 0$ if and only if $\mathrm{nsupp}(\bm{v}+\bm{u}-\bm{u'})$. In particular, $[\bm{v}+B\bm{s}+\bm{u}]_{\bm{u}_{+}} \neq 0$ if and only if $\mathrm{nsupp}(\bm{v})\subset \mathrm{supp}(B)\cup \mathrm{nsupp}(\bm{v}+\bm{u})$.
%\end{lem}
Let $\bm{w}$ be a generic weight. Recall that $\bm{v}$ is called a \textit{fake exponent} of $H_{A}(\bm{\beta})$ with respect to $\bm{w}$ if $A\bm{v}=\bm{\beta}$ and $[\bm{v}]_{\bm{u}_{+}}=0$ for all $\bm{u}\in L$ with $\bm{u}_{+}\cdot\bm{w} >\bm{u}_{-}\cdot\bm{w}$, where $\bm{u} \cdot \bm{w} = \sum_{j=1}^{n}u_{j}w_{j}.$\par
From now on, fix a generic weight $\bm{w}$, a fake exponent $\bm{v}$ of $H_{A}(\bm{\beta})$ with respect to $\bm{w}$.\par
We abbreviate nsupp$(\bm{v}+\bm{u})$ to $I_{\bm{u}}$ for $\bm{u}\in L$. In particular, $I_{\bm{0}}=\mathrm{nsupp}(\bm{v})$.\par
For $\bm{u}\in L$, let
\begin{equation*}
    a_{\bm{u}}(\bm{s}):=\frac{[\bm{v}+B\bm{s}]_{\bm{u}_{-}}}{[\bm{v}+B\bm{s}+\bm{u}]_{\bm{u}_{+}}}.
\end{equation*}
Throughout this paper, we put the following assumption.
\begin{ass}\label{ass2}
    $B$ is a basis for $L$ over $\mathbb{Z}$. Hence, $h=n-d$.
\end{ass}

\begin{lem}\textup{{\cite[Lemma 2.4]{log19}}}\label{2.4}
    Let $\bm{u}\in L$. Under Assumption \ref{ass2}, $a_{\bm{u}}(\bm{s})\neq 0$.
\end{lem}
We recall some definitions of concepts related to the fake exponent $\bm{v}$.\par
Let
\begin{equation*}
    \mathcal{G}:=\{\partial^{\bm{g}_{+}^{(i)}}-\partial^{\bm{g}_{-}^{(i)}}\mid i=1,\ldots,m\}
\end{equation*}
be the reduced Gr\"{o}bner basis of $I_{A}$ with respect to $\bm{w}$ with $\partial^{\bm{g}_{+}^{(i)}} \in \mathrm{in}_{\bm{w}}I_{A}$ for all $i$. Set
\begin{equation*}
    C(\bm{w}):=\sum_{i=1}^{m}\mathbb{N}\bm{g}^{(i)}.
\end{equation*}
A collection $\mathrm{NS}_{\bm{w}}(\bm{v})$ of negative supports $I_{\bm{u}}\; (\bm{u}\in L)$ is defined by
\begin{equation*}
    \mathrm{NS}_{\bm{w}}(\bm{v}):=\{ I_{\bm{u}} \mid \bm{u}\in L.\; \mathrm{If} \; I_{\bm{u}}=I_{\bm{u'}} \; \mathrm{for} \; \bm{u'} \in L , \; \mathrm{then} \; \bm{u'}\in C(\bm{w})\}.
\end{equation*}
In addition, define
\begin{equation*}
    \mathrm{NS}_{\bm{w}}(\bm{v})^{c}:=\{I_{\bm{u}} \mid \bm{u}\in L\} \setminus \mathrm{NS}_{\bm{w}}(\bm{v}).
\end{equation*}\par
Consider the subset $L'$ of $L$ defined by
\begin{equation*}
    L':=\{ \bm{u}\in L \mid I_{\bm{u}}\in \mathrm{NS}_{w}(\bm{v})\}.
\end{equation*}
By definition, we see that $L'\subset C(\bm{w})$. \par
Let
\begin{equation*}
    K:=\bigcap_{I\in \mathrm{NS}_{w}(\bm{v})} I,
\end{equation*}
and define the homogeneous ideal $P$ of $\mathbb{C}[\bm{s}]$ as
\begin{equation*}
    P:=\Big\langle (B\bm{s})^{I\cup J \setminus K} \mid I\in \mathrm{NS}_{w}(\bm{v}), J\in \mathrm{NS}_{w}(\bm{v})^{c}\Big\rangle.
\end{equation*}
In addition, we define the orthogonal complement $P^{\perp}$ for $P\subset \mathbb{C}[\bm{s}]$ as
\begin{eqnarray*}
    P^{\perp}&:=&\{q(\partial_{\bm{s}})\in \mathbb{C}[\partial_{\bm{s}}] \mid (q(\partial_{\bm{s}})\bullet h(\bm{s}))_{|\bm{s}=0} =0 \; \mathrm{for \; all}\; h(\bm{s})\in P\}\\
    &=&\{q(\partial_{\bm{s}})\in \mathbb{C}[\partial_{\bm{s}}] \mid q(\partial_{\bm{s}})\bullet P \subset \langle s_{1},\ldots, s_{h}\rangle \},
\end{eqnarray*}
where $\mathbb{C}[\partial_{\bm{s}}]:=\mathbb{C}[\partial_{s_{1}},\ldots,\partial_{s_{h}}]$. Since $P$ and $\langle s_{1},\ldots, s_{h} \rangle$ are both homogeneous, $P^{\perp}$ is homogeneous with respect to the usual total ordering.

We define
\begin{equation*}
    m(\bm{s}):=(B{\bm{s}})^{I_{\bm{0}}\setminus K},
\end{equation*}
\begin{equation*}
    F(\bm{x},\bm{s}):=\sum_{\bm{u}\in L'}a_{\bm{u}}(\bm{s})x^{\bm{v}+B\bm{s}+\bm{u}},
\end{equation*}
and
\begin{equation*}
    \tilde F(\bm{x},\bm{s}):=m(\bm{s})F(\bm{x},\bm{s}).
\end{equation*}

Next, we define $G^{(i)}$ for $i=1,\ldots,m$ and $P_{B}$ as
\begin{equation*}
    G^{(i)}:=I_{-\bm{g}^{(i)}}\setminus I_{\bm{0}},
\end{equation*}
\begin{equation*}
    P_{B}:=\langle (B\bm{s})^{G^{(i)}}\mid i=1,\ldots m\rangle.
\end{equation*}

In \cite{log19}, some relations between $P$ and $P_{B}$ are shown. For instance, \cite[Proposition 3.4]{log19}:
    \begin{equation*}
        m(\bm{s}) \cdot P_{B} \subset P \subset P_{B}.
    \end{equation*}
    In particular, if $K=I_{\bm{0}}$, then $P=P_{B}$.
    
For $U(\partial_{\bm{z}})\in\mathbb{C}[\partial_{\bm{z}}]$ and $q(\partial_{\bm{s}})\in\mathbb{C}[\partial_{\bm{s}}]$, a $\mathbb{C}$-linear operation $U(\partial_{z_{1}},\ldots,\partial_{z_{h}})\star q(\partial_{\bm{s}})$ is defined by
\begin{equation*}
U(\partial_{z_{1}},\ldots,\partial_{z_{h}})\star q(\partial_{\bm{s}}):=(U(\partial_{\bm{z}})\bullet q(\bm{z}))|_{\bm{z}=\partial_{\bm{s}}}\in\mathbb{C}[\partial_{\bm{s}}].
\end{equation*}
This is the star operation in \cite{log19}. For the star operator, we have
    \begin{equation*}
    m(\bm{s})\star P^\perp \subset P_{B}^{\perp} \subset P^{\perp}.    
    \end{equation*}
 In particular, if $K=I_{\bm{0}}$, then $P^{\perp}=P_{B}^{\perp}$.

%\begin{cor}\textup{{\cite[Corollary 4.3]{log19}}}\label{4.3}
%Assume that $B$ is a basis of $L$. If $|I\cup J|>|I_{0}|$ for any $I\in \mathcal{N}$ and $J\in \mathcal{N}^{c}$, then $\bm{v}$ is an exponent.
%\end{cor}
\begin{thm}\textup{{\cite[Theorem 4.4]{log19}}}\label{4.4}
    Assume that $B$ is a basis of $L_{\mathbb{C}}$, and that $P=m(\bm{s})\cdot P_{B}$. Then $\bm{v}$ is an exponent, and the set
    \begin{equation*}
    \{ q(\partial_{\bm{s}})\bullet \tilde{F}(\bm{x},\bm{s})_{|\bm{s}=\bm{0}} \mid q(\partial_{\bm{s}})\in P^{\perp}\}
    \end{equation*}
spans the space of series solutions to $M_{A}(\bm{\beta})$ with exponent $\bm{v}$ in the direction $\bm{w}$.
\end{thm}

%\begin{cor}\textup{[refine Corollary \ref{4.3}]}
    %Assume that $B$ is a basis of $L$. If $|I\cup J|>|I_{0}|=0$ for any $I\in %\mathrm{NS}_{\bm{w}}(\bm{v})$ and $J\in \mathrm{NS}_{\bm{w}}(\bm{v})^{c}$, then %$\bm{v}$ is an exponent.
%\end{cor}
\section{Main Theorem}
We review standrad pairs from \cite{saito6}.\par
Let $M$ be a monomial ideal in $\mathbb{C}[\bm{\partial}]$. A \textit{standard pair} of $M$ is a pair $(\bm{a},\sigma)$, where $\bm{a}\in\mathbb{N}^{n}$ and $\sigma\subset \{1,\ldots,n\}$ subject to the following three conditions:
    \begin{enumerate}
        \item $a_{i}=0$ for all $i\in\sigma$;
        \item for all choices of integers $b_{j}\geq 0$, the monomial $\partial^{\bm{a}}\cdot \prod_{j\in\sigma}\partial_{j}^{b_{j}}$ is not in $M$;
        \item For any $l\notin \sigma$, there exists $b_{j}\geq 0$ such that $\partial^{\bm{a}}\cdot \partial_{l}^{b_{l}}\cdot \prod_{j\in\sigma}\partial_{j}^{b_{j}}$ lies in $M$.
    \end{enumerate}\par
    Let $S(M)$ denote the set of standard pairs of $M$. Then, by \cite[Corollary 3.2.3]{SST}, $\bm{v}$ is a fake exponent of $M_{A}(\bm{\beta})$ with respect to $\bm{w}$ if and only if $A\bm{v}=\bm{\beta}$ and there exists a standard pair $(\bm{a},\sigma)\in S(\mathrm{in}_{\bm{w}}I_{A})$ such that $v_{j}=a_{j}$ for all $j\notin\sigma$.
\newline

Put
\begin{align*}
\Delta_{\bm{w}}:&=\{\sigma\mid (\bm{a},\sigma)\in S(\mathrm{in}_{\bm{w}}I_{A})\},\\
    C_{\bm{w}}:&=\cap_{\sigma\in\Delta_{\bm{w}}}\sigma\quad (\subset\{1,\ldots,n\}).
\end{align*}

\begin{thm}\label{mainth}
    If $I_{\bm{0}}=\mathrm{nsupp}(\bm{v})\subset C_{\bm{w}}$, then $I_{\bm{0}}=K$. 
\end{thm}
\begin{proof}
By definition, $K\subset I_{\bm{0}}$. We prove the statement by contradiction. Suppose $j\in I_{\bm{0}}\setminus K$. Then $j\in C_{\bm{w}}$ by the assumption. From \cite[(3.13)]{SST},
    $$
    \mathrm{in}_{\bm{w}}I_{A}=\bigcap_{(\bm{a},\sigma)\in S(\mathrm{in}_{\bm{w}}I_{A})}\langle \partial_{i}^{a_{i}+1}\mid i\notin \sigma\rangle.
    $$
    Since $j\in C_{\bm{w}}$, we see that $j\in\sigma$ for all $(\bm{a},\sigma)\in S(\mathrm{in}_{\bm{w}}I_{A})$. Therefore in the system of minimal generators of $\mathrm{in}_{\bm{w}}I_{A}$, $\partial_j$ does not appear.
    Let $\mathcal{G}=\{\partial^{\bm{g}_{+}^{(i)}}-\partial^{\bm{g}_{-}^{(i)}}\mid i=1,\ldots,m\}$ with $\bm{w}\cdot \bm{g}_{+}^{(i)}>\bm{w}\cdot \bm{g}_{-}^{(i)}$ be the reduced Gr\"{o}bner basis of $I_{A}$ with respect to $\bm{w}$. By definition, $\{\partial^{\bm{g}_{+}^{(i)}}\mid i=1,\ldots, m\}$ equals the system of minimal generators of $\mathrm{in}_{\bm{w}}I_{A}$. Hence all $\bm{g}^{(i)}$ satisfy $g_{j}^{(i)}\leq0$. Now since $j\notin K$, there exists a vector $\sum_{i=1}^{m}x_{i}\bm{g}^{(i)}$ such that $x_{1}g_{j}^{(1)}+\cdots +x_{m}g_{j}^{(m)}+v_{j}\geq 0$, where $x_{1},\ldots,x_{k}\in \mathbb{N}$. However it is impossible, since we know $v_{j}<0$ from $j\in I_{\bm{0}}=\mathrm{nsupp}({\bm{v}})$. Thus we obtain a contradiction. 
\end{proof}
\begin{rem}
If Theorem \ref{mainth} is applicable, then $I_{0}=K$ holds, and a basis of the solutions space corresponding to $\bm{v}$ can be obtained simply by computing $P_{B}$.
\end{rem}

We apply Theorem \ref{mainth} to the following examples: Example \ref{3.2} and Example \ref{3.3}. In Example \ref{3.2}, we can obtain a basis for the space of series solutions by $P_{B}$. In example \ref{3.3}, there exist both a basis for the solution space obtained by $P_{B}$, and bases that can be obtained without any computation.
\begin{ex}\textup{\cite[Example 4.2.7]{SST}}\label{3.2}
    Let 
    $$
    A=
    \begin{bmatrix}
        1 & 1 & 1 & 1 & 1 & 1\\
        0 & 1 & 1 & 0 & -1 & -1\\
        -1 & -1 & 0 & 1 & 1 & 0
    \end{bmatrix}.
    $$
\begin{diagram}
\begin{center}
\begin{tikzpicture}
\draw[very thin,color=gray] (-2.5,-2.5) grid (2.5,2.5);
\draw[color=blue, ->] (-3.2,0) -- (3.2,0) node[above] {};%¡²ix²j
\draw[color=blue, ->] (0,-3.2) -- (0,3.2) node[right] {};%c²
\draw[color=red] (0, -1) -- (1, -1);
\draw[color=red] (1, -1) -- (1, 0);
\draw[color=red] (0, 1) -- (1, 0);
\draw[color=red] (-1, 0) -- (0, -1);
\draw[color=red] (0, 1) -- (-1, 1);
\draw[color=red] (-1, 1) -- (-1, 0);

\coordinate (A) at (0,-1) node at (A) [below right] {$\bm{a}_{1}$}; 
\fill (A) circle (2pt); 
\coordinate (B) at (1,-1) node at (B) [above right] {$\bm{a}_{2}$}; 
\fill (B) circle (2pt); 

\node (C) at (3.5, 0) {$y$};
\node (D) at (0, 3.5) {$z$};
\coordinate (E) at (1,0) node at (E) [above right] {$\bm{a}_{3}$}; 
\fill (E) circle (2pt); 
\coordinate (F) at (0,1) node at (F) [above right] {$\bm{a}_{4}$}; 
\fill (F) circle (2pt); 
\coordinate (G) at (-1,1) node at (G) [below left] {$\bm{a}_{5}$}; 
\fill (G) circle (2pt); 
\coordinate (H) at (-1,0) node at (H) [below left] {$\bm{a}_{6}$}; 
\fill (H) circle (2pt); 
\node (I) at (3.5, 3.5) {$(x=1)$};
\end{tikzpicture}
\end{center}
\end{diagram}
\begin{comment}
There are the following relations;
\begin{eqnarray*}
&&
\bm{a}_{1}+\bm{a}_{4}=\bm{a}_{2}+\bm{a}_{5}=\bm{a}_{3}+\bm{a}_{6},\\
&&
\bm{a}_{1}+\bm{a}_{3}+\bm{a}_{5}=\bm{a}_{2}+\bm{a}_{4}+\bm{a}_{6},\\
&&
\bm{a}_{1}+2\bm{a}_{3}=2\bm{a}_{2}+\bm{a}_{4},\qquad
2\bm{a}_{4}+\bm{a}_{6}=2\bm{a}_{5}+\bm{a}_{3},\\
&&
2\bm{a}_{1}+\bm{a}_{3}=2\bm{a}_{2}+\bm{a}_{6},\qquad
\bm{a}_{4}+2\bm{a}_{6}=2\bm{a}_{5}+\bm{a}_{1},\\
&&
2(\bm{a}_{1}+\bm{a}_{5})=\bm{a}_{3}+3\bm{a}_{6},\qquad
2(\bm{a}_{2}+\bm{a}_{6})=\bm{a}_{4}+3\bm{a}_{1},\\
&&
2(\bm{a}_{3}+\bm{a}_{5})=\bm{a}_{1}+3\bm{a}_{4},\qquad
2(\bm{a}_{2}+\bm{a}_{4})=\bm{a}_{6}+3\bm{a}_{3},\\
&&
2\bm{a}_{1}+\bm{a}_{5}=2\bm{a}_{6}+\bm{a}_{2},\qquad
2\bm{a}_{3}+\bm{a}_{5}=2\bm{a}_{4}+\bm{a}_{2}.
\end{eqnarray*}
\end{comment}
    Let $\bm{w}=(2,3,8,1,13,5)$. Then
    \begin{align*}
        I_{A}=\langle &\partial_{2}\partial_{5}-\partial_{1}\partial_{4},\;\partial_{3}\partial_{6}-\partial_{1}\partial_{4},\;\partial_{3}\partial_{5}^{2}-\partial_{4}^{2}\partial_{6},\;\partial_{3}^{2}\partial_{5}-\partial_{2}\partial_{4}^{2},\; \partial_{1}\partial_{5}^{2}-\partial_{4}\partial_{6}^{2},\\
        &\partial_{1}\partial_{3}\partial_{5}-\partial_{2}\partial_{4}\partial_{6},\;\partial_{1}\partial_{3}^{2}-\partial_{2}^{2}\partial_{4},\;\partial_{1}^{2}\partial_{5}-\partial_{2}\partial_{6}^{2},\;\partial_{2}^{2}\partial_{6}^{2}-\partial_{1}^{3}\partial_{4},\;\partial_{1}^{2}\partial_{3}-\partial_{2}^{2}\partial_{6} \rangle.
    \end{align*}
and
$$
\mathrm{in}_{\bm{w}} I_{A}
=
\langle
\partial_2\partial_5,\; \partial_3\partial_6,\;
\partial_3\partial_5^2,\;
\partial_3^2\partial_5,\;
\partial_1\partial_5^2,\;
\partial_1\partial_3\partial_5,\;
\partial_1\partial_3^2,\;
\partial_1^2\partial_5,\;
\partial_2^2\partial_6^2,\;
\partial_1^2\partial_3\;
\rangle.
$$
The standard pairs of $\mathrm{in}_{\bm{w}} I_{A}$ of the top dimension are
$$
\begin{array}{cccc}
(0, 0, 0, \ast, \ast, \ast),
&
(\ast,0,0,\ast,0,\ast),
&
(\ast,\ast,0,\ast,0,0),
&
(0,\ast,\ast,\ast,0,0),\\
&
(\ast,1,0,\ast,0,\ast),
&
(\ast,\ast,0,\ast,0,1).
\end{array}
$$
The embedded standard pairs are
$$
(1,\ast,1,\ast,0,0),\quad
(1,0,0,\ast,1,\ast),\quad
(0,0,1,\ast,1,0),
$$
and $C_{\bm{w}}=\{4\}$.\par
Let $\bm{\beta}=c\bm{a}_{4}$. Then there exist 2 fake exponents $\bm{v}_{1}=(0,0,0,c,0,0)$ and $\bm{v}_{2}=(-\frac{3}{2},1,0,c-\frac{1}{2},0,1)$; $\bm{v}_{1}$ corresponds to standard pairs $(0,0,0,\ast,\ast,\ast)$, $(\ast,0,0,\ast,0,\ast)$, $(\ast,\ast,0,\ast,0,0)$ and $(0,\ast,\ast,\ast,0,0)$, and $\bm{v}_{2}$ to $(\ast,1,0,\ast,0,\ast)$ and $(\ast,\ast,0,\ast,0,1)$. Since $\mathrm{nsupp}(\bm{v}_{1}),\:\mathrm{nsupp}(\bm{v}_{2})\subset C_{\bm{w}}=\{4\}$, we have $I_{\bm{0}}\subset K$ at each fake exponent by Theorem \ref{mainth}.

\end{ex}
\begin{ex}\label{3.3}
Let $A$ be the same as Example \ref{3.2}. Let $\bm{w}=(5,3,1,0,0,0)$. Then
\begin{align*}
        I_{A}=\langle &\partial_{2}\partial_{5}-\partial_{3}\partial_{6},\;\partial_{1}\partial_{4}-\partial_{3}\partial_{6},\;\partial_{3}\partial_{5}^{2}-\partial_{4}^{2}\partial_{6},\;\partial_{2}\partial_{4}^{2}-\partial_{3}^{2}\partial_{5},\; \partial_{1}\partial_{5}^{2}-\partial_{4}\partial_{6}^{2},\\
        &\partial_{1}\partial_{3}\partial_{5}-\partial_{2}\partial_{4}\partial_{6},\;\partial_{1}\partial_{3}^{2}-\partial_{2}^{2}\partial_{4},\;\partial_{1}^{2}\partial_{5}-\partial_{2}\partial_{6}^{2},\;\partial_{1}^{2}\partial_{3}-\partial_{2}^{2}\partial_{6} \rangle.
    \end{align*}
and
    $$
    \mathrm{in}_{\bm{w}}I_{A}=\langle  \partial_{2}\partial_{5},\;\partial_{1}\partial_{4},\;\partial_{3}\partial_{5}^{2},\;\partial_{2}\partial_{4}^{2},\; \partial_{1}\partial_{5}^{2},\;\partial_{1}\partial_{3}\partial_{5},\;\partial_{1}\partial_{3}^{2},\;\partial_{1}^{2}\partial_{5},\;\partial_{1}^{2}\partial_{3} \rangle.
    $$
    The standard pairs of $
    \mathrm{in}_{\bm{w}}I_{A}$ of top dimension are
    $$
\begin{array}{cccc}
    (\ast,\ast,0,0,0,\ast),
    &
    (0,0,\ast,\ast,0,\ast),
    &
    (0,\ast,\ast,0,0,\ast),
    &
    (0,0,0,\ast,\ast,\ast),
    \\
    &
    (0,0,\ast,\ast,1,\ast),
    &
    (0,\ast,\ast,1,0,\ast).
    \end{array}
    $$
    The embedded standard pairs are
$$
(1,\ast,1,0,0,\ast),\quad
(1,0,0,0,1,\ast).
$$
    In this case,
    $$
        \Delta =\{ \{1,2,6\},\{2,3,6\},\{3,4,6\},\{4,5,6\},\{2,6\},\{6\}\},
    $$
    and $C_{\bm{w}}=\{6\}$.\par
    Let $\bm{\beta}=(-1,1,0)=-\bm{a}_{6}$. Then there exist 3 fake exponents $\bm{v}_{1}=(0,0,0,0,0,-1)$, $\bm{v}_{2}=(0,1,-\frac{3}{2},1,0,-\frac{3}{2})$ and $\bm{v}_{3}=(0,0,\frac{1}{2},-1,1,-\frac{3}{2})$; $\bm{v}_{1}$ corresponds to standard pairs $(\ast,\ast,0,0,0,\ast)$, $(0,\ast,\ast,0,0,\ast)$, $(0,0,\ast,\ast,0,\ast)$ and $(0,0,0,\ast,\ast,\ast)$, $\bm{v}_{2}$ to $(0,\ast,\ast,1,0,\ast)$, and $\bm{v}_{3}$ to $(0,0,\ast,\ast,1,\ast)$. Since $\mathrm{nsupp}(\bm{v}_{1})=\{6\}=C_{\bm{w}}$, $\mathrm{nsupp}(\bm{v}_{2})=\emptyset\subset C_{\bm{w}}$ and $\mathrm{nsupp}(\bm{v}_{4})=\{4\}\not\subset C_{\bm{w}}$, we have $I_{\bm{0}}\subset K$ at fake exponents $\bm{v}_{1}$ and  $\bm{v}_{2}$ by Theorem \ref{mainth}. 
\end{ex}
We generalize Example \ref{3.2} and \ref{3.3}.
\begin{ex}
In Example \ref{3.2}, \ref{3.3}, we considered a hexagon, but now we consider an $n$-gon. Let $A=[\bm{a}_{1},\ldots,\bm{a}_{n}]$, where $\bm{a}_{i}$ is the vartex of an $n$-gon, and let $\bm{w}$ is the order with $w_{1}<w_{n}<w_{n-1}<\cdots<w_{3}<w_{2}$, as illustrated in the Figure 2.
\begin{diagram}
\begin{center}
\begin{tikzpicture}
\draw[color=red] (2,2) -- (1.19534,2.56343);
\draw[color=red,dashed] (1.19534,2.56343) -- (0.24651,2.81766);
\draw[color=red] (2,2) -- (2.56343, 1.19534);
\draw[color=red,dashed] (2.56343, 1.19534) -- (2.81766,0.24651);
\draw[color=red] (-2,-2) -- (-1.19534, -2.56343);
\draw[color=red] (-2,-2) -- (-2.56343,-1.19534);
\draw[color=red,dashed] (-1.19534, -2.56343) -- (-0.24651,-2.81766);
\draw[color=red,dashed] (-2.56343,-1.19534) -- (-2.81766,-0.24651);

\coordinate (A) at (2,2) node at (A) [above right] {$\bm{a}_{1}$}; 
\fill (A) circle (2pt); 
\coordinate (B) at (1.19534,2.56343) node at (B) [above right] {$\bm{a}_{2}$}; 
\fill (B) circle (2pt); 

\coordinate (E) at (2.56343, 1.19534) node at (E) [above right] {$\bm{a}_{3}$}; 
\fill (E) circle (2pt); 
\coordinate (F) at (-1.19534, -2.56343) node at (F) [below] {$\bm{a}_{n-1}$}; 
\fill (F) circle (2pt); 
\coordinate (G) at (-2.56343,-1.19534) node at (G) [below left] {$\bm{a}_{n-2}$}; 
\fill (G) circle (2pt); 
\coordinate (H) at (-2,-2) node at (H) [below left] {$\bm{a}_{n}$}; 
\fill (H) circle (2pt); 

\draw (2,2)--(0.24651,2.81766);
\draw (2,2)--(2.81766,0.24651);
\draw (2,2)--(-0.24651,-2.81766);
\draw (2,2)--(-2.81766,-0.2465);
\draw(2,2)--(-1.19534, -2.56343);
\draw(2,2)--(-2.56343,-1.19534);
\draw(2,2)--(-2,-2);

\end{tikzpicture}
\end{center}
\end{diagram}
In this case, $C_{\bm{w}}=\{1\}$. If $\bm{\beta}=-m\bm{a}_{1},\;m\in\mathbb{Z}_{>0}$, then there exists a fake expoent $v=(-m,0,\ldots,0)$. Hence we obtain $\mathrm{nsupp}(\bm{v})=\{1\}=C_{\bm{w}}$ and we have $I_{\bm{0}}\subset K$ by Theorem \ref{mainth}. 
\end{ex}
Under Assumption \ref{ass2}, we obtain the following corollary from Theorem \ref{4.4} and Theorem \ref{mainth}.
\begin{cor}
If $\mathrm{nsupp}(\bm{v})\subset C_{\bm{w}}$. Then the set
\begin{equation*}
    \{(q(\partial_{\bm{s}}) \bullet F(\bm{x},\bm{s}))_{|\bm{s}=\bm{0}} \mid q(\partial_{\bm{s}})\in P_{B}^{\perp}\}
\end{equation*}
equals the space of series solutions to $M_{A}(\bm{\beta})$ with exponent $\bm{v}$ in the direction $\bm{w}$.
\end{cor}

\begin{dfn}
    The triangulation $\Delta_{\bm{w}}$ determined by $\bm{w}$ is said to be unimodular if all facets $\sigma$ of $\Delta_{\bm{w}}$ have normalized volume equal to 1. The matrix $A$ is said to be unimodular if every square regular matrix has determinant $+1$ or $-1$. If $A$ is unimodular, then any triangulation is unimodular \cite[p.139]{SST}.
\end{dfn}
\begin{rem}
If $A$ is unimodular, then
$$\{ \mathrm{facets \: of \:} \Delta_{\bm{w}}\}=\{\{\bm{a}_{j}\mid j\in \sigma\}\mid \sigma\in\Delta_{\bm{w}}\}.
$$
\end{rem}
\begin{cor}\label{NT}
    Assume that $\Delta_{\bm{w}}$ is unimodular. If  $\bm{\beta}\in\mathrm{Core}(\Delta_{\bm{w}})$, then $\bm{\beta}$ has a unique fake exponent $\bm{v}$ with respect to $\bm{w}$, $\mathrm{supp}(\bm{v})\subset C_{\bm{w}}$, and at $\bm{v}$ we have $I_{\bm{0}}=K$.
    In particular, if $\bm{\beta} \in \mathrm{Core}(\Delta_{\bm{w}})$, we can construct the space of $A$-hypergeometric series in the direction $\bm{w}$ by Frobenius's method.
\end{cor}
\begin{proof}
If $\Delta_w$ is unimodular, then all standard pairs are formed only $0$ and $\ast$. Hence $\bm{v}$ has to satisfy $\mathrm{supp}(\bm{v})\subset C_{\bm{w}}$ as in the proof of Lemma 5.1 in \cite{log19}.
\end{proof}

\section{Aomoto-Gel'fand system of type $3\times 3$}
Let $\{\bm{e}_{1},\ldots,\bm{e}_{M+N}\}$ be the standard basis of $\mathbb{Z}^{M+N}$, and put
$$
\bm{a}_{i,j}:=\bm{e}_{i}+\bm{e}_{j}\qquad (1\leq i \leq M < j \leq M+N),
$$
$$
A=[\bm{a}_{1,M+1},\ldots,\bm{a}_{1,M+N},\bm{a}_{2,M+1},\ldots,\bm{a}_{2,M+N},\ldots,\bm{a}_{d,M+1},\ldots,\bm{a}_{M+N} ].
$$
Then $n=MN,\:d=\mathrm{rank}(A)=M+N-1$, and by \cite[Example 5.1]{sturmfels}, the toric ideal is
$$
I_{A}=\langle \partial_{ij}\partial_{kl}-\partial_{il}\partial_{kj}\mid 1\leq i,k\leq M<j,l\leq M+N \rangle .
$$
According to \cite[p.139]{SST}, the configuration $A$ is the vertex set of a product of two simplices and is unimodular.\par
We consider the Aomoto-Gel'fand of type $3\times 3$ for $\bm{v}$ such that $\mathrm{nsupp}(\bm{v})\subset C_{\bm{w}}$. Then
\begin{equation*}
    A = \begin{bmatrix}
    1 & 1 & 1 & 0 & 0 & 0 & 0 & 0 & 0\\
    0 & 0 & 0 & 1 & 1 & 1 & 0 & 0 & 0\\
    0 & 0 & 0 & 0 & 0 & 0 & 1 & 1 & 1\\
    1 & 0 & 0 & 1 & 0 & 0 & 1 & 0 & 0\\
    0 & 1 & 0 & 0 & 1 & 0 & 0 & 1 & 0\\
    0 & 0 & 1 & 0 & 0 & 1 & 0 & 0 & 1\\
    \end{bmatrix}.
\end{equation*}
A generic weight $\bm{w}$ induces one of the 5 triangulations illustrated by A.\:E.\:Postnikov in \cite[p.250]{black}, up to $S_{3}\times S_{3}$ symmetry. The following is one of the triangulations of the $\Delta^{3}\times\Delta^{3}$ where $\ast$ denotes a vertex of each simplex.\\
Case1 (staircase)
\begin{equation*}
    \begin{bmatrix}
        0 & 0 & \ast \\
        0 & 0 & \ast \\
        \ast & \ast & \ast
    \end{bmatrix}, \;
    \begin{bmatrix}
        0 & 0 & \ast \\
        0 & \ast & \ast \\
        \ast & \ast & 0
    \end{bmatrix}, \;
    \begin{bmatrix}
        0 & \ast & \ast \\
        0 & \ast & 0 \\
        \ast & \ast & 0
    \end{bmatrix}, \;
    \begin{bmatrix}
        0 & 0 & \ast \\
        \ast & \ast & \ast \\
        \ast & 0 & 0
    \end{bmatrix}, \;
    \begin{bmatrix}
        0 & \ast & \ast \\
        \ast & \ast & 0 \\
        \ast & 0 & 0
    \end{bmatrix}, \;
    \begin{bmatrix}
        \ast & \ast & \ast \\
        \ast & 0 & 0 \\
        \ast & 0 & 0
    \end{bmatrix}.
\end{equation*}
Case 2
\begin{equation*}
    \begin{bmatrix}
        0 & 0 & \ast \\
        0 & \ast & 0\\
        \ast & \ast & \ast
    \end{bmatrix}, \;
    \begin{bmatrix}
        0 & \ast & \ast \\
        0 & \ast & 0 \\
        \ast & \ast & 0
    \end{bmatrix}, \;
    \begin{bmatrix}
        0 & 0 & \ast \\
        0 & \ast & \ast \\
        \ast & 0 & \ast
    \end{bmatrix}, \;
    \begin{bmatrix}
        0 & 0 & \ast \\
        \ast & \ast & \ast \\
        \ast & 0 & 0
    \end{bmatrix}, \;
    \begin{bmatrix}
        0 & \ast & \ast \\
        \ast & \ast & 0 \\
        \ast & 0 & 0
    \end{bmatrix}, \;
    \begin{bmatrix}
        \ast & \ast & \ast \\
        \ast & 0 & 0 \\
        \ast & 0 & 0
    \end{bmatrix}.
\end{equation*}
Case 3
\begin{equation*}
    \begin{bmatrix}
        \ast & \ast & \ast \\
        0 & \ast & 0\\
        0 & \ast & 0
    \end{bmatrix}, \;
    \begin{bmatrix}
        \ast & 0 & \ast \\
        0 & \ast & 0 \\
        \ast & \ast & 0
    \end{bmatrix}, \;
    \begin{bmatrix}
        0 & 0 & \ast \\
        0 & \ast & \ast \\
        \ast & \ast & 0
    \end{bmatrix}, \;
    \begin{bmatrix}
        \ast & 0 & \ast \\
        0 & \ast & \ast \\
        \ast & 0 & 0
    \end{bmatrix}, \;
    \begin{bmatrix}
        0 & 0 & \ast \\
        0 & 0 & \ast \\
        \ast & \ast & \ast
    \end{bmatrix}, \;
    \begin{bmatrix}
        \ast & 0 & 0 \\
        \ast & \ast & \ast \\
        \ast & 0 & 0
    \end{bmatrix}.
\end{equation*}
Case 4
\begin{equation*}
    \begin{bmatrix}
        0 & \ast & \ast \\
        0 & \ast & 0\\
        \ast & \ast & 0
    \end{bmatrix}, \;
    \begin{bmatrix}
        0 & 0 & \ast \\
        0 & \ast & \ast \\
        \ast & \ast & 0
    \end{bmatrix}, \;
    \begin{bmatrix}
        0 & 0 & \ast \\
        0 & 0 & \ast \\
        \ast & \ast & \ast
    \end{bmatrix}, \;
    \begin{bmatrix}
        \ast & \ast & \ast \\
        0 & \ast & 0 \\
        \ast & 0 & 0
    \end{bmatrix}, \;
    \begin{bmatrix}
        \ast & 0 & \ast \\
        0 & \ast & \ast \\
        \ast & 0 & 0
    \end{bmatrix}, \;
    \begin{bmatrix}
        \ast & 0 & 0 \\
        \ast & \ast & \ast \\
        \ast & 0 & 0
    \end{bmatrix}.
\end{equation*}
Case 5
\begin{equation*}
    \begin{bmatrix}
        0 & 0 & \ast \\
        0 & \ast & 0\\
        \ast & \ast & \ast
    \end{bmatrix}, \;
    \begin{bmatrix}
        0 & \ast & \ast \\
        0 & \ast & 0 \\
        \ast & \ast & 0
    \end{bmatrix}, \;
    \begin{bmatrix}
        \ast & \ast & \ast \\
        0 & \ast & 0 \\
        \ast & 0 & 0
    \end{bmatrix}, \;
    \begin{bmatrix}
        0 & 0 & \ast \\
        0 & \ast & \ast \\
        \ast & 0 & \ast
    \end{bmatrix}, \;
    \begin{bmatrix}
        0 & 0 & \ast \\
        \ast & \ast & \ast \\
        \ast & 0 & 0
    \end{bmatrix}, \;
    \begin{bmatrix}
        \ast & 0 & \ast \\
        \ast & \ast & 0 \\
        \ast & 0 & 0
    \end{bmatrix}.
\end{equation*}
\begin{align*}
    \mathrm{Case}\:1:\Delta=\{&\{3,6,7,8,9\},\{3,5,6,7,8\},\{2,3,5,7,8\},\{3,4,5,6,7\},\\
    &\{2,3,4,5,7\},\{1,2,3,4,7\}\},\:C_{w}=\{3,7\},\\
    \mathrm{Core}(\Delta_{\bm{w}})=&\{\mathbb{C}\bm{a}_{3}+\mathbb{C}\bm{a}_{7}\}.
    \end{align*}
    \begin{align*}
    \mathrm{Case}\:2:\Delta=\{&\{3,5,7,8,9\},\{2,3,5,7,8\},\{3,5,6,7,9\},\{3,4,5,6,7\},\\
    &\{2,3,4,5,7\},\{1,2,3,4,7\}\},\:C_{w}=\{3,7\},\\
    \mathrm{Core}(\Delta_{\bm{w}})=&\{\mathbb{C}\bm{a}_{3}+\mathbb{C}\bm{a}_{7}\}.
    \end{align*}
    \begin{align*}
    \mathrm{Case}\:3:\Delta=\{&\{1,2,3,5,8,\},\{1,3,5,7,8\},\{3,5,6,7,8\},\{1,3,5,6,7\},\\
    &\{3,6,7,8,9\},\{1,4,5,6,7\}\},\:C_{w}=\emptyset,\\
    \mathrm{Core}(\Delta_{\bm{w}})=&\emptyset.
    \end{align*}
    \begin{align*}
    \mathrm{Case}\:4:\Delta=\{&\{2,3,5,7,8\},\{3,5,6,7,8\},\{3,6,7,8,9\},\{1,2,3,5,7\},\\
    &\{1,3,5,6,7\},\{1,4,5,6,7\}\},\:C_{w}=\{7\},\\
    \mathrm{Core}(\Delta_{\bm{w}})=&\{\mathbb{C}\bm{a}_{7}\}.
    \end{align*}
    \begin{align*}
    \mathrm{Case}\:5:\Delta=\{&\{3,5,7,8,9\},\{2,3,5,7,8\},\{1,2,3,5,7\},\{3,5,6,7,9\},\\
    &\{3,4,5,6,7\},\{1,3,4,5,7\}\},\:C_{w}=\{3,5,7\},\\
    \mathrm{Core}(\Delta_{\bm{w}})=&\{\mathbb{C}\bm{a}_{3}+\mathbb{C}\bm{a}_{5}+\mathbb{C}\bm{a}_{7}\}.
    \end{align*}
\\
Case 1\par
Let $\bm{w}=(1,2,3,5,8,13,21,34,55)$. Then
\begin{equation*}
    \mathrm{in}_{\bm{w}}I_{A} = \langle
    \partial_{1}\partial_{5},\:
    \partial_{1}\partial_{6},\:
    \partial_{2}\partial_{6},\:
    \partial_{1}\partial_{8},\:
    \partial_{4}\partial_{8},\:
    \partial_{1}\partial_{9},\:
    \partial_{2}\partial_{9},\:
    \partial_{4}\partial_{9},\:
    \partial_{5}\partial_{9}
    \rangle.
\end{equation*}
This weight $\bm{w}$ induces a staircase regular triangulation and the standard pairs are
\begin{equation*}
    \begin{bmatrix}
        0 & 0 & \ast \\
        0 & 0 & \ast \\
        \ast & \ast & \ast
    \end{bmatrix}, \;
    \begin{bmatrix}
        0 & 0 & \ast \\
        0 & \ast & \ast \\
        \ast & \ast & 0
    \end{bmatrix}, \;
    \begin{bmatrix}
        0 & \ast & \ast \\
        0 & \ast & 0 \\
        \ast & \ast & 0
    \end{bmatrix}, \;
    \begin{bmatrix}
        0 & 0 & \ast \\
        \ast & \ast & \ast \\
        \ast & 0 & 0
    \end{bmatrix}, \;
    \begin{bmatrix}
        0 & \ast & \ast \\
        \ast & \ast & 0 \\
        \ast & 0 & 0
    \end{bmatrix}, \;
    \begin{bmatrix}
        \ast & \ast & \ast \\
        \ast & 0 & 0 \\
        \ast & 0 & 0
    \end{bmatrix}.
\end{equation*}
In this case, the reduced Gr\"{o}bner basis $\mathcal{G}$ of $I_{A}$ with respect to $\bm{w}$ is
\begin{equation*}
\begin{split}
    \mathcal{G} = \{
    &\underline{\partial_{1}\partial_{5}}-\partial_{2}\partial_{4},
    \underline{\partial_{1}\partial_{8}}-\partial_{2}\partial_{7},
    \underline{\partial_{4}\partial_{8}}-\partial_{5}\partial_{7},\\
    &\underline{\partial_{1}\partial_{6}}-\partial_{3}\partial_{4},
    \underline{\partial_{1}\partial_{9}}-\partial_{3}\partial_{7},
    \underline{\partial_{4}\partial_{9}}-\partial_{6}\partial_{7},\\
    &\underline{\partial_{2}\partial_{6}}-\partial_{3}\partial_{5},
    \underline{\partial_{2}\partial_{9}}-\partial_{3}\partial_{8},
    \underline{\partial_{5}\partial_{9}}-\partial_{6}\partial_{8}
    \},
\end{split}
\end{equation*}
where the underlined terms are leading ones on the weight $\bm{w}$. Put
\begin{eqnarray*}
    g^{(1)} &=& (1,\;-1,\;0,\;-1,\;1,\;0,\;0,\;0,\;0)^{T},\\
    g^{(2)} &=& (1,\;-1,\;0,\;0,\;0,\;0,\;-1,\;1,\;0)^{T},\\
    g^{(3)} &=& (0,\;0,\;0,\;1,\;-1,\;0,\;-1,\;1,\;0)^{T},\\
    g^{(4)} &=& (1,\;0,\;-1,\;-1,\;0,\;1,\;0,\;0,\;0)^{T},\\
    g^{(5)} &=& (1,\;0,\;-1,\;0,\;0,\;0,\;-1,\;0,\;1)^{T},\\
    g^{(6)} &=& (0,\;0,\;0,\;1,\;0,\;-1,\;-1,\;0,\;1)^{T},\\
    g^{(7)} &=& (0,\;1,\;-1,\;0,\;-1,\;1,\;0,\;0,\;0)^{T},\\
    g^{(8)} &=& (0,\;1,\;-1,\;0,\;0,\;0,\;0,\;-1,\;1)^{T},\\
    g^{(9)} &=& (0,\;0,\;0,\;0,\;1,\;-1,\;0,\;-1,\;1)^{T}.
\end{eqnarray*}
Since the following equality holds:
\begin{eqnarray*}
g^{(2)} &=&   g^{(1)} + g^{(3)},\\
        g^{(4)} &=&   g^{(1)} + g^{(7)},\\
        g^{(5)} &=&   g^{(1)} + g^{(3)} + g^{(7)} + g^{(9)},\\
        g^{(6)} &=&  g^{(3)} + g^{(9)},\\
        g^{(8)} &=&  g^{(7)} + g^{(9)},\\
    \end{eqnarray*}
    we have
\begin{eqnarray*}
    C(\bm{w})&=&\sum_{i=1}^{9} \mathbb{N} g^{(i)}\\
    &=&\mathbb{N}g^{(1)} \oplus  \mathbb{N}g^{(3)}
    \oplus  \mathbb{N}g^{(7)} \oplus  \mathbb{N}g^{(9)}.
\end{eqnarray*}
        
Put $B=\{g^{(1)},\:g^{(3)},\:g^{(7)},\:g^{(9)}\}$ and $s=(s_{1},\;s_{2},\;s_{3},\;s_{4})$, then
\begin{equation*}
    Bs = \begin{bmatrix}
        s_{1}\\
        -s_{1}+s_{3}\\
        -s_{3}\\
        -s_{1}+s_{2}\\
        s_{1}-s_{2}-s_{3}+s_{4}\\
        s_{3}-s_{4}\\
        -s_{2}\\
        s_{2}-s_{4}\\
        s_{4}
        \end{bmatrix}.
\end{equation*}
Let $\bm{\beta}=\bm{0}$. Then $\bm{v}=\bm{0}$, and $I_{\bm{0}}=\emptyset$. Hence $G^{(i)}=I_{-\bm{g}^{(i)}}=\mathrm{nsupp}(-\bm{g}^{(i)})$ and we obtain
\begin{align*}
    \mathrm{nsupp}(-g^{(1)})&=\{1,5\},\\
    \mathrm{nsupp}(-g^{(2)})&=\{1,8\},\\
    \mathrm{nsupp}(-g^{(3)})&=\{4,8\},\\
    \mathrm{nsupp}(-g^{(4)})&=\{1,6\},\\
    \mathrm{nsupp}(-g^{(5)})&=\{1,9\},\\
    \mathrm{nsupp}(-g^{(6)})&=\{4,9\},\\
    \mathrm{nsupp}(-g^{(7)})&=\{2,6\},\\
    \mathrm{nsupp}(-g^{(8)})&=\{2,9\},\\
    \mathrm{nsupp}(-g^{(9)})&=\{5,9\}.\\
\end{align*}
Furthermore $P_{B}$ and $P^{\perp}_{B}$ are
    \begin{equation*}
    \begin{split}
    P_{B} &= \langle
    s_{1}(s_{1}-s_{2}-s_{3}+s_{4}), \;
    s_{1}(s_{2}-s_{4}), \;
    (-s_{1}+s_{2})(s_{2}-s_{4}), \;
    s_{1}(s_{3}-s_{4}), \;
    s_{1}s_{4}, \;\\
    &\quad\quad(-s_{1}+s_{2})s_{4}, \;
    (-s_{1}+s_{3})(s_{3}-s_{4}), \;
    (-s_{1}+s_{3})s_{4}, \;
    (s_{1}-s_{2}-s_{3}+s_{4})s_{4} \rangle\\
    &= \langle
    s^2_{1},\;s^2_{2},\;s^2_{3},\;s^2_{4},\; s_{1}s_{2},\; s_{1}s_{3},\; s_{1}s_{4},\; s_{2}s_{4},\; s_{3}s_{4}\rangle,
    \end{split}
    \end{equation*}
    \begin{equation*}
        P^{\perp}_{B}=\langle 1,\; \partial_{s_{1}},\; \partial_{s_{2}},\; \partial_{s_{3}},\; \partial_{s_{4}},\; \partial_{s_{2}}\partial_{s_{3}}
        \rangle.
    \end{equation*}
Then, we see the solution space of this case is spanned by
\begin{equation*}
    \begin{split}
        \{&(F(x,s))_{|s=0},\; \partial_{s_{1}}(F(x,s))_{|s=0},\; \partial_{s_{2}}(F(x,s))_{|s=0},\\
        &\partial_{s_{3}}(F(x,s))_{|s=0},\; \partial_{s_{4}}(F(x,s))_{|s=0},\; \partial_{s_{2}}\partial_{s_{3}}(F(x,s))_{|s=0}\}.
    \end{split}
\end{equation*}
For a general $\bm{\beta}\in\mathrm{Core}(\Delta_{\bm{w}})$, we similarly obtain the solution space.\\
\\
For the other cases, we can do the same as follows. For simplicity, we suppose that $\bm{\beta}=\bm{0}$, and hence $\bm{v}=\bm{0}$.\\
\\
Case 2\par
Let $\bm{w}=(55,34,1,21,2,13,3,5,8)$. Then 
\begin{equation*}
    \mathrm{in}_{\bm{w}}I_{A} = \langle
    \partial_{1}\partial_{5},\:
    \partial_{1}\partial_{6},\:
    \partial_{2}\partial_{6},\:
    \partial_{1}\partial_{8},\:
    \partial_{4}\partial_{8},\:
    \partial_{1}\partial_{9},\:
    \partial_{2}\partial_{9},\:
    \partial_{4}\partial_{9},\:
    \partial_{6}\partial_{8}
    \rangle.
\end{equation*}
This $\bm{w}$ induces the following triangulation (Case 2).
\begin{equation*}
    \begin{bmatrix}
        0 & 0 & \ast \\
        0 & \ast & 0\\
        \ast & \ast & \ast
    \end{bmatrix}, \;
    \begin{bmatrix}
        0 & \ast & \ast \\
        0 & \ast & 0 \\
        \ast & \ast & 0
    \end{bmatrix}, \;
    \begin{bmatrix}
        0 & 0 & \ast \\
        0 & \ast & \ast \\
        \ast & 0 & \ast
    \end{bmatrix}, \;
    \begin{bmatrix}
        0 & 0 & \ast \\
        \ast & \ast & \ast \\
        \ast & 0 & 0
    \end{bmatrix}, \;
    \begin{bmatrix}
        0 & \ast & \ast \\
        \ast & \ast & 0 \\
        \ast & 0 & 0
    \end{bmatrix}, \;
    \begin{bmatrix}
        \ast & \ast & \ast \\
        \ast & 0 & 0 \\
        \ast & 0 & 0
    \end{bmatrix}.
\end{equation*}
In this case, the solution space of this case is spanned by
\begin{equation*}
    \begin{split}
        \{&(F(x,s))_{|s=0},\; \partial_{s_{1}}(F(x,s))_{|s=0},\; \partial_{s_{2}}(F(x,s))_{|s=0},\;
        \partial_{s_{3}}(F(x,s))_{|s=0},\\ &\partial_{s_{4}}(F(x,s))_{|s=0},\; (\partial_{s_{2}}\partial_{s_{3}}+\partial_{s_{2}}\partial_{s_{4}}+\partial_{s_{3}}\partial_{s_{4}}-\frac{1}{2}\partial^{2}_{s_{4}})(F(x,s))_{|s=0}\}.
    \end{split}
\end{equation*}\\
\\
Case 3 \par
Let $\bm{w}=(1,55,2,34,3,5,8,13,21)$. Then
\begin{equation*}
    \mathrm{in}_{\bm{w}}I_{A} = \langle
    \partial_{2}\partial_{4},\:
    \partial_{2}\partial_{7},\:
    \partial_{4}\partial_{8},\:
    \partial_{3}\partial_{4},\:
    \partial_{1}\partial_{9},\:
    \partial_{4}\partial_{9},\:
    \partial_{2}\partial_{6},\:
    \partial_{2}\partial_{9},\:
    \partial_{5}\partial_{9} \rangle.
\end{equation*}
This $\bm{w}$ induces the following triangulation (Case 3).
\begin{equation*}
    \begin{bmatrix}
        \ast & \ast & \ast \\
        0 & \ast & 0\\
        0 & \ast & 0
    \end{bmatrix}, \;
    \begin{bmatrix}
        \ast & 0 & \ast \\
        0 & \ast & 0 \\
        \ast & \ast & 0
    \end{bmatrix}, \;
    \begin{bmatrix}
        0 & 0 & \ast \\
        0 & \ast & \ast \\
        \ast & \ast & 0
    \end{bmatrix}, \;
    \begin{bmatrix}
        \ast & 0 & \ast \\
        0 & \ast & \ast \\
        \ast & 0 & 0
    \end{bmatrix}, \;
    \begin{bmatrix}
        0 & 0 & \ast \\
        0 & 0 & \ast \\
        \ast & \ast & \ast
    \end{bmatrix}, \;
    \begin{bmatrix}
        \ast & 0 & 0 \\
        \ast & \ast & \ast \\
        \ast & 0 & 0
    \end{bmatrix}.
\end{equation*}
In this case, the solution space of this case is spanned by
\begin{equation*}
    \begin{split}
        \{&(F(x,s))_{|s=0},\; \partial_{s_{1}}(F(x,s))_{|s=0},\; \partial_{s_{2}}(F(x,s))_{|s=0},\\
        &\partial_{s_{3}}(F(x,s))_{|s=0},\; \partial_{s_{4}}(F(x,s))_{|s=0},\; \frac{1}{2}\partial^{2}_{s_{4}}(F(x,s))_{|s=0}\}.
    \end{split}
\end{equation*}\\
\\
Case 4 \par
Let $\bm{w}=(21,13,1,55,2,3,5,8,34)$. Then
\begin{equation*}
    \mathrm{in}_{\bm{w}}I_{A} = \langle
    \partial_{2}\partial_{4},\:
    \partial_{1}\partial_{8},\:
    \partial_{4}\partial_{8},\:
    \partial_{3}\partial_{4},\:
    \partial_{1}\partial_{9},\:
    \partial_{4}\partial_{9},\:
    \partial_{2}\partial_{6},\:
    \partial_{2}\partial_{9},\:
    \partial_{5}\partial_{9}
    \rangle.
\end{equation*}
This $\bm{w}$ induces the following triangulation (Case 4).
\begin{equation*}
    \begin{bmatrix}
        0 & \ast & \ast \\
        0 & \ast & 0\\
        \ast & \ast & 0
    \end{bmatrix}, \;
    \begin{bmatrix}
        0 & 0 & \ast \\
        0 & \ast & \ast \\
        \ast & \ast & 0
    \end{bmatrix}, \;
    \begin{bmatrix}
        0 & 0 & \ast \\
        0 & 0 & \ast \\
        \ast & \ast & \ast
    \end{bmatrix}, \;
    \begin{bmatrix}
        \ast & \ast & \ast \\
        0 & \ast & 0 \\
        \ast & 0 & 0
    \end{bmatrix}, \;
    \begin{bmatrix}
        \ast & 0 & \ast \\
        0 & \ast & \ast \\
        \ast & 0 & 0
    \end{bmatrix}, \;
    \begin{bmatrix}
        \ast & 0 & 0 \\
        \ast & \ast & \ast \\
        \ast & 0 & 0
    \end{bmatrix}.
\end{equation*}
In this case, the solution space of this case is spanned by
\begin{equation*}
    \begin{split}
        \{&(F(x,s))_{|s=0},\; \partial_{s_{1}}(F(x,s))_{|s=0},\; \partial_{s_{2}}(F(x,s))_{|s=0},\;
        \partial_{s_{3}}(F(x,s))_{|s=0},\\ &\partial_{s_{4}}(F(x,s))_{|s=0},\; (\partial_{s_{1}}\partial_{s_{3}}-\frac{1}{2}\partial^{2}_{s_{3}})(F(x,s))_{|s=0}\}.
    \end{split}
\end{equation*}\\
\\
Case 5\par
Let $\bm{w}=(2,1,0,4,0,5,0,2,5)$. Then
\begin{equation*}
    \mathrm{in}_{\bm{w}}I_{A} = \langle
    \partial_{2}\partial_{4},\:
    \partial_{1}\partial_{8},\:
    \partial_{4}\partial_{8},\:
    \partial_{1}\partial_{6},\:
    \partial_{1}\partial_{9},\:
    \partial_{4}\partial_{9},\:
    \partial_{2}\partial_{6},\:
    \partial_{2}\partial_{9},\:
    \partial_{6}\partial_{8}
    \rangle.
\end{equation*}
Then $\bm{w}$ induces the following triangulation (Case 5).
\begin{equation*}
    \begin{bmatrix}
        0 & 0 & \ast \\
        0 & \ast & 0\\
        \ast & \ast & \ast
    \end{bmatrix}, \;
    \begin{bmatrix}
        0 & \ast & \ast \\
        0 & \ast & 0 \\
        \ast & \ast & 0
    \end{bmatrix}, \;
    \begin{bmatrix}
        \ast & \ast & \ast \\
        0 & \ast & 0 \\
        \ast & 0 & 0
    \end{bmatrix}, \;
    \begin{bmatrix}
        0 & 0 & \ast \\
        0 & \ast & \ast \\
        \ast & 0 & \ast
    \end{bmatrix}, \;
    \begin{bmatrix}
        0 & 0 & \ast \\
        \ast & \ast & \ast \\
        \ast & 0 & 0
    \end{bmatrix}, \;
    \begin{bmatrix}
        \ast & 0 & \ast \\
        \ast & \ast & 0 \\
        \ast & 0 & 0
    \end{bmatrix}.
\end{equation*}
In this case, space of this case is spanned by
\begin{equation*}
    \begin{split}
        \{&(F(x,s))_{|s=0},\; \partial_{s_{1}}(F(x,s))_{|s=0},\; \partial_{s_{2}}(F(x,s))_{|s=0},\;
        \partial_{s_{3}}(F(x,s))_{|s=0},\\ &\partial_{s_{4}}(F(x,s))_{|s=0},\; (\frac{1}{2}\partial^{2}_{s_{1}}-\partial_{s_{2}}\partial_{s_{3}}+\frac{1}{2}\partial^{2}_{s_{4}})(F(x,s))_{|s=0}\}.
    \end{split}
\end{equation*}

\subsection{Application of the Main Theorem to K3 Surfaces}
So far, we have discussed the case $\bm{\beta}=\bm{0}$ for simplicity. However, our method also applies to the case $\bm{\beta}=
(1/2, 1/2, 1/2, 1/2, 1/2, 1/2)$. This $\bm{\beta}$ arises in the study a period map for a family of K3 surfaces \cite{Love}. Here, we investigate the applicability of the main theorem to this parameter. The following cases are discussed in page 11.\\
\\
Case 1\par
In this case, $\bm{\beta}\not\in\mathrm{Core}(\Delta_{\bm{w}})$. However, there exists fake exponents $v_{1}=(0,0,1/2,0,0,1/2,-1/2,1/2,1/2),\:v_{2}=(-1/2,1/2,1/2,1/2,0,0,1/2,0,0),$\\$v_{3}=(0,0,1/2,0,1/2,0,1/2,0,0)$ with multiplicity $4$. Since $\mathrm{nsupp}(v_{3})=\emptyset\subset C_{w}$, we obtain $I_{0}=K$ by Theorem \ref{mainth}.\\
\\
Case 5\par
In this case, $\bm{\beta}\in\mathrm{Core}(\Delta_{\bm{w}})$. Therefore, we have $I_{0}=K$ by Corollary \ref{NT}. Indeed, there exists a unique fake exponent $v=(0,0,1/2,0,0,1/2,0,0,1/2)$ with multiplicity $6$. Since $\mathrm{nsupp}(v)=\emptyset\subset C_{w}$, we obtain $I_{0}=K$ by Theorem \ref{mainth}.\\
\\
Theorem \ref{mainth} can also be applied to the other cases, just as in Case 1.

\subsection*{\normalsize Mao NAGAMINE}
\noindent
Department of Mathematics \\
Graduate School of Science\\
Hokkaido University\\
Sapporo 060-0810, JAPAN\\
e-mail: nagamine.mao.d9@elms.hokudai.ac.jp
\end{document}